\documentclass[12pt]{amsart}
\usepackage{tikz}
\usepackage{array}
\usepackage{float}
\usepackage{multirow}
\usepackage{amsthm,amsmath,amsfonts,caption}
\usepackage[top=1.25in,bottom=1in,left=1in,right=1in]{geometry}

\newtheorem{theorem}{Theorem}
\newtheorem{lemma}{Lemma}
\newtheorem{proposition}{Proposition}
\newtheorem*{mattreethm}{Kirchhoff's Matrix-Tree Theorem}
\newtheorem*{matrixdetlemma}{The Matrix Determinant Lemma}
 
\DeclareMathOperator{\adj}{adj}
\newcommand{\vu}{\mathbf{u}}
\newcommand{\vv}{\mathbf{v}}

\begin{document}

\title{Linear Algebraic Techniques for Spanning Tree Enumeration}
\markright{Counting Spanning Trees}
\author{Steven Klee and Matthew T. Stamps}

\maketitle

\begin{abstract}
Kirchhoff's matrix-tree theorem asserts that the number of spanning trees in a finite graph can be computed from the determinant of any of its reduced Laplacian matrices.  In many cases, even for well-studied families of graphs, this can be computationally or algebraically taxing.  We show how two well-known results from linear algebra, the matrix determinant lemma and the Schur complement, can be used to count the spanning trees in several significant families of graphs in an elegant manner.  
\end{abstract}

\section{Introduction}

A \textbf{graph} $G$ consists of a finite set of vertices and a set of edges that connect some pairs of vertices.  For the purposes of this article, we will assume that all graphs are simple, meaning they do not contain loops (an edge connecting a vertex to itself) or multiple edges between a given pair of vertices.  We will use $V(G)$ and $E(G)$ to denote the vertex set and edge set of $G$, respectively.    For example, the graph $G$ with \begin{center} $V(G) = \{1,2,3,4\}$ \quad and \quad $E(G) = \{\{1,2\},\{2,3\},\{3,4\},\{1,4\},\{1,3\}\}$ \end{center} is shown in Figure \ref{spanning-tree-non-examples}.  

A \textbf{spanning tree} in a graph $G$ is a subgraph $T \subseteq G$, meaning $T$ is a graph with $V(T) \subseteq V(G)$ and $E(T) \subseteq E(G)$, that satisfies three conditions: 
\begin{enumerate}\itemsep 5pt
\item Every vertex in $G$ is a vertex in $T$; 
\item $T$ is connected, meaning it is possible to walk between any two vertices in $G$ using only edges in $T$; and 
\item $T$ does not contain any cycles.
\end{enumerate}
Figure \ref{spanning-tree-non-examples} shows an example and several nonexamples of spanning trees in our example graph $G$ to illustrate the relevance of each of these conditions. 

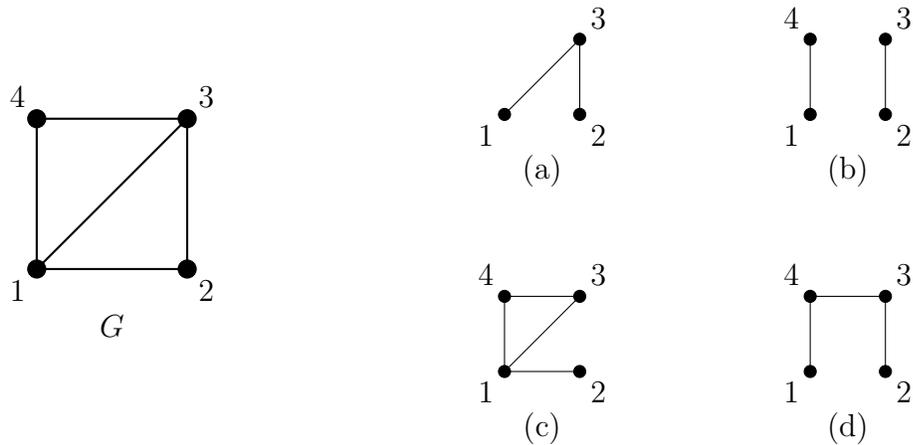
\begin{figure}[H]
\begin{tabular}{>{\centering\arraybackslash}m{.425\textwidth}>{\centering\arraybackslash}m{.225\textwidth}>{\centering\arraybackslash}m{.225\textwidth}}
\multirow{2}{*}[2pt]{
\begin{tikzpicture}
\foreach \v in {(0,0), (2,0), (2,2), (0,2)}{
\draw[fill=black] \v circle (.12);
}
\draw [thick] (0,0) node[anchor = north east] {$1$}  -- (2,0) node[anchor = north west] {$2$} -- (2,2) node[anchor = south west] {$3$} -- (0,2) node[anchor = south east] {$4$} -- (0,0) -- (2,2);
\draw (1,-0.75) node {$G$};
\end{tikzpicture}
}

&
\begin{tikzpicture}
\foreach \v in {(0,0), (1,0), (1,1)}{
\draw[fill=black] \v circle (.08);
}
\draw (0,0) node[anchor = north east] {$1$};
\draw (1,0) node[anchor = north west] {$2$};
\draw (1,1) node[anchor = south west] {$3$};

\draw (0,0) -- (1,1) -- (1,0);
\draw(0.5,-0.75) node {(a)};
\end{tikzpicture}

&
\begin{tikzpicture}
\foreach \v in {(0,0), (1,0), (1,1), (0,1)}{
\draw[fill=black] \v circle (.08);
}
\draw (0,0) node[anchor = north east] {$1$};
\draw (1,0) node[anchor = north west] {$2$};
\draw (1,1) node[anchor = south west] {$3$};
\draw (0,1) node[anchor = south east] {$4$};

\draw (0,0) -- (0,1);
\draw (1,0) -- (1,1);
\draw(0.5,-0.75) node {(b)};
\end{tikzpicture}

\\
& 
\begin{tikzpicture}
\foreach \v in {(0,0), (1,0), (1,1), (0,1)}{
\draw[fill=black] \v circle (.08);
}
\draw (0,0) node[anchor = north east] {$1$};
\draw (1,0) node[anchor = north west] {$2$};
\draw (1,1) node[anchor = south west] {$3$};
\draw (0,1) node[anchor = south east] {$4$};

\draw (0,0) -- (0,1) -- (1,1) -- (0,0) -- (1,0);
\draw(0.5,-0.75) node {(c)};
\end{tikzpicture}

& 
\begin{tikzpicture}
\foreach \v in {(0,0), (1,0), (1,1), (0,1)}{
\draw[fill=black] \v circle (.08);
}
\draw (0,0) node[anchor = north east] {$1$};
\draw (1,0) node[anchor = north west] {$2$};
\draw (1,1) node[anchor = south west] {$3$};
\draw (0,1) node[anchor = south east] {$4$};

\draw (0,0) -- (0,1) -- (1,1) -- (1,0);

\draw(0.5,-0.75) node {(d)};
\end{tikzpicture}

\end{tabular}
\vspace{-10pt}
\caption{A graph $G$ (left) and four of its subgraphs (right).  Subgraph (a) is not a spanning tree because it does not include all the vertices of $G$; subgraph (b) is not a spanning tree because it is not connected; and subgraph (c) is not a spanning tree because it contains a cycle on vertices $1$, $3$, and $4$. Subgraph (d), on the other hand, is a spanning tree. }
\label{spanning-tree-non-examples}
\end{figure}

For our example graph $G$, it is possible to write down all the spanning trees.  This is shown in Figure \ref{spanning-tree-list}.  

\begin{figure}[H]
\begin{tabular}{>{\centering\arraybackslash}m{.225\textwidth}>{\centering\arraybackslash}m{.225\textwidth}>{\centering\arraybackslash}m{.225\textwidth}>{\centering\arraybackslash}m{.225\textwidth}}

\begin{tikzpicture}
\foreach \v in {(0,0), (1,0), (1,1), (0,1)}{
\draw[fill=black] \v circle (.08);
}
\draw (0,0) node[anchor = north east] {$1$};
\draw (1,0) node[anchor = north west] {$2$};
\draw (1,1) node[anchor = south west] {$3$};
\draw (0,1) node[anchor = south east] {$4$};

\draw (0,0) -- (0,1) -- (1,1) -- (1,0);
\end{tikzpicture}

&

\begin{tikzpicture}
\foreach \v in {(0,0), (1,0), (1,1), (0,1)}{
\draw[fill=black] \v circle (.08);
}
\draw (0,0) node[anchor = north east] {$1$};
\draw (1,0) node[anchor = north west] {$2$};
\draw (1,1) node[anchor = south west] {$3$};
\draw (0,1) node[anchor = south east] {$4$};

\draw (0,1) -- (1,1) -- (1,0) -- (0,0);
\end{tikzpicture}

&

\begin{tikzpicture}
\foreach \v in {(0,0), (1,0), (1,1), (0,1)}{
\draw[fill=black] \v circle (.08);
}
\draw (0,0) node[anchor = north east] {$1$};
\draw (1,0) node[anchor = north west] {$2$};
\draw (1,1) node[anchor = south west] {$3$};
\draw (0,1) node[anchor = south east] {$4$};

\draw (1,1) -- (1,0) -- (0,0) -- (0,1);
\end{tikzpicture}

&

\begin{tikzpicture}
\foreach \v in {(0,0), (1,0), (1,1), (0,1)}{
\draw[fill=black] \v circle (.08);
}
\draw (0,0) node[anchor = north east] {$1$};
\draw (1,0) node[anchor = north west] {$2$};
\draw (1,1) node[anchor = south west] {$3$};
\draw (0,1) node[anchor = south east] {$4$};

\draw (1,0) -- (0,0) -- (0,1) -- (1,1);

\end{tikzpicture}

\\

\begin{tikzpicture}
\foreach \v in {(0,0), (1,0), (1,1), (0,1)}{
\draw[fill=black] \v circle (.08);
}
\draw (0,0) node[anchor = north east] {$1$};
\draw (1,0) node[anchor = north west] {$2$};
\draw (1,1) node[anchor = south west] {$3$};
\draw (0,1) node[anchor = south east] {$4$};

\draw (0,1) -- (0,0) -- (1,1) -- (1,0);
\end{tikzpicture}

&

\begin{tikzpicture}
\foreach \v in {(0,0), (1,0), (1,1), (0,1)}{
\draw[fill=black] \v circle (.08);
}
\draw (0,0) node[anchor = north east] {$1$};
\draw (1,0) node[anchor = north west] {$2$};
\draw (1,1) node[anchor = south west] {$3$};
\draw (0,1) node[anchor = south east] {$4$};

\draw (0,1) -- (0,0) -- (1,0);
\draw (0,0) -- (1,1);
\end{tikzpicture}

&

\begin{tikzpicture}
\foreach \v in {(0,0), (1,0), (1,1), (0,1)}{
\draw[fill=black] \v circle (.08);
}
\draw (0,0) node[anchor = north east] {$1$};
\draw (1,0) node[anchor = north west] {$2$};
\draw (1,1) node[anchor = south west] {$3$};
\draw (0,1) node[anchor = south east] {$4$};
\draw (0,1) -- (1,1) -- (1,0);
\draw (0,0) -- (1,1);

\end{tikzpicture}

&

\begin{tikzpicture}
\foreach \v in {(0,0), (1,0), (1,1), (0,1)}{
\draw[fill=black] \v circle (.08);
}
\draw (0,0) node[anchor = north east] {$1$};
\draw (1,0) node[anchor = north west] {$2$};
\draw (1,1) node[anchor = south west] {$3$};
\draw (0,1) node[anchor = south east] {$4$};

\draw (0,1) -- (1,1) -- (0,0) -- (1,0);
\end{tikzpicture}

\end{tabular}
\vspace{-10pt}
\caption{All the spanning trees in the graph $G$ from Figure~\ref{spanning-tree-non-examples}.}
\label{spanning-tree-list}
\end{figure}
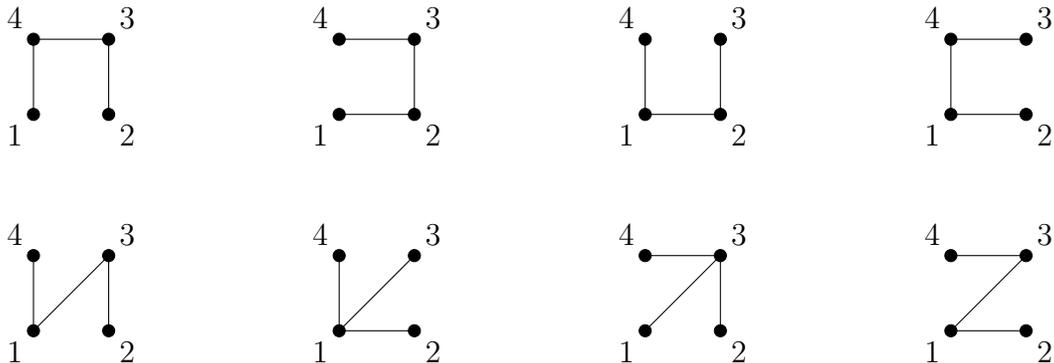

In general, the number of spanning trees in a graph can be quite large, and exhaustively listing all of its spanning trees is not feasible.  For this reason, we need to be more resourceful when counting the spanning trees in a graph. Throughout this article, we will use $\tau(G)$ to denote the number of spanning trees in a graph $G$.  A common approach for calculating $\tau(G)$ involves linear algebraic techniques and the Laplacian matrix of $G$. 

If $G$ is a graph with vertex set $V = V(G)$, the \textbf{Laplacian matrix} of $G$, denoted by $L(G)$ is a $|V|\times |V|$ matrix whose rows and columns are indexed by vertices of $G$.  The entries of the Laplacian matrix are given by 

$$
L(G)(i,j) = 
\begin{cases}
\deg(i) & \text{ if } i=j, \\
-1 & \text{ if } i \neq j \text{ and } \{i,j\} \in E(G), \\
0 & \text{ otherwise.}
\end{cases}
$$ \smallskip

Here, $\deg(i)$ denotes the \textbf{degree} of vertex $i$ in $G$, which is the number of edges in $G$ that contain $i$.  It follows from the definition that the rows and columns of $L(G)$ sum to $0$, which means $L(G)$ is always singular.  Remarkably, if we delete any row and any column of $L(G)$, then up to a sign, the determinant of the resulting matrix counts the number of spanning trees in $G$.  This is the content of Kirchhoff's matrix-tree theorem \cite{Kirchhoff}, which we state more precisely below.  Let $L(G)_{\ell,m}$ denote the matrix obtained from $L(G)$ by deleting the $\ell$th row and $m$th column. This submatrix is sometimes called a \textbf{reduced Laplacian} of $G$.

\begin{mattreethm}
Let $G$ be a connected graph with vertex set $V(G) = \{1,\ldots,n\}$.  For any vertices $\ell,m \in V(G)$, not necessarily distinct, $$\tau(G) = (-1)^{\ell+m}\det(L(G)_{\ell,m}).$$ 
\end{mattreethm}

The Laplacian matrix of the graph $G$ from Figure \ref{spanning-tree-non-examples}, along with the reduced Laplacian $L(G)_{3,2}$ are
$$
L(G) = 
\left[
\begin{array}{rrrr}
3 & -1 & -1 & -1 \\
-1 & 2 & -1 & 0 \\
-1 & -1 & 3 & -1 \\
-1 & 0 & -1 & 2 
\end{array}
\right]
\quad \text{ and } \quad
L(G)_{3,2} = 
\left[
\begin{array}{rrr}
3  & -1 & -1 \\
-1 & -1 & 0 \\
-1 & -1 & 2 
\end{array}
\right].
$$ \smallskip

It can be checked that $\det(L(G)) = 0$ and $\det\left(L(G)_{3,2}\right) = -8 = (-1)^{3+2}\cdot 8$, which is consistent with the eight spanning trees we listed in Figure \ref{spanning-tree-list}.

For the remainder of this article, we will combine two well-known results in linear algebra with the matrix-tree theorem to obtain elegant explicit formulas for $\tau(G)$ in terms of the number and degrees of its vertices.

\section{Tools from linear algebra}

Even for simple graphs such as the one in Figure \ref{spanning-tree-non-examples}, computing the determinant of a reduced Laplacian can be computationally intensive.  Depending on the choice of which row $\ell$ and column $m$ are to be crossed out, the determinant of some reduced Laplacians might be easier to compute than others, but the necessity to make a such a choice may seem unsatisfactory.  In this section, we present our main result, which will play a central role throughout the subsequent sections of the article.  We proceed with the first of two well-known results from linear algebra. Recall that the \textbf{adjugate} of an $n \times n$ matrix is the transpose of its $n \times n$ matrix of cofactors.

\begin{matrixdetlemma}
Let $M$ be an $n \times n$ matrix and let $\vu$ and $\vv$ be column vectors in $\mathbb{R}^n$.  Then $$\det(M+\vu\vv^T) = \det(M) + \vv^T\adj(M)\vu. $$  In particular, if $M$ is invertible, then $\det(M+\vu\vv^T) = \det(M)\left(1+\vv^TM^{-1}\vu\right)$.
\end{matrixdetlemma}

A proof of the matrix determinant lemma can be found in \cite[\S0.8.5]{Horn-Johnson}, where it is referred to as \emph{Cauchy's formula for the determinant of a rank-one perturbation}.  The main ingredients in the proof are the fact that $\det(\cdot)$ is a multilinear operator on the rows of a matrix and the observation that $\vu\vv^T$ is a rank-one matrix. 
With this, we are ready to present our main result. 

\begin{lemma} \label{rank-one-update}
Let $G$ be a graph on vertex set $V$ with Laplacian matrix $L$ and let $\mathbf{u}=(u_i)_{i \in V}$ and $ \mathbf{v}=(v_i)_{i \in V}$ be column vectors in $\mathbb{R}^V$.  Then $$\det(L + \mathbf{u}\mathbf{v}^T) = \left(\sum_{i \in V} u_i \right) \cdot \left(\sum_{i \in V} v_i \right) \cdot \tau(G).$$
\end{lemma}

\begin{proof}
Let $\mathbf{1}_{V, V}$ denote the $|V| \times |V|$ matrix of ones and let $\mathbf{1}_V$ denote the $|V|\times1$ column vector of ones.  
By Kirchoff's matrix-tree theorem, every cofactor of $L$ is equal to $\tau(G)$, so $\adj(L)$ can be written as $\tau(G) \mathbf{1}_{V, V} = \tau(G) \mathbf{1}_V \mathbf{1}_V^T.$ Therefore, by the matrix determinant lemma, 
\begin{eqnarray*}
\det(L + \mathbf{u}\mathbf{v}^T) &=& \det(L) + \mathbf{v}^T \adj(L) \mathbf{u} \\[0.5em]
&=& 0 + \mathbf{v}^T\left(\tau(G) \mathbf{1}_V \mathbf{1}_V^T\right)\mathbf{u} \\[0.5em]
&=& \left(\mathbf{v}^T\mathbf{1}_V\right)\left( \mathbf{1}_V^T\mathbf{u}\right) \cdot \tau(G)\\[0.5em]
&=& \left(\sum_{i \in V} v_i \right) \cdot \left(\sum_{i \in V} u_i \right) \cdot \tau(G).
\end{eqnarray*}
\end{proof}

In the case that $\vu\vv^T = \mathbf{1}_{V, V}$, this result is credited to Temperley \cite{Temperley}. 

\subsection{The Schur complement of a matrix.}

Later in this article, there will be instances in which we partition the vertices of a graph into disjoint subsets as $V(G) = V_1 \cup V_2$.  In such instances, the Laplacian matrix of $G$ can be decomposed into a block matrix of the form 

$$\begin{bmatrix} A & B \\ C & D \end{bmatrix},$$ \vspace{-1pt}

where the first $|V_1|$ rows and columns correspond to the vertices in $V_1$ and the last $|V_2|$ rows and columns correspond to the vertices in $V_2$.  In this case, $A$ and $D$ are square matrices of sizes $|V_1| \times |V_1|$ and $|V_2| \times |V_2|$, respectively.  

Suppose $M$ is any square matrix that can be decomposed into blocks $A, B, C, D$ as above with $A$ and $D$ square.  If $D$ is invertible, then the \textbf{Schur complement} of $D$ in $M$ is defined as $M/D:= A- BD^{-1}C$.  The fundamental reason for using Schur complements is the following result in \cite[\S0.8.5]{Horn-Johnson}.

\begin{lemma} \label{schur-complement}
Let $M$ be a square matrix decomposed as above into blocks $A$, $B$, $C$, and $D$ with $A$ and $D$ square and $D$ invertible.  Then $$\det(M) = \det(D) \cdot \det(A-BD^{-1}C).$$
\end{lemma}

When $M = \begin{bmatrix} a & b \\ c & d \end{bmatrix}$ is a $2 \times 2$ matrix and $d \neq 0$, Lemma~\ref{schur-complement} simply says  $$\det(M) = d \left(a - b \cdot \frac{1}{d} \cdot c\right),$$ which is just an alternative way of writing the familiar formula for the determinant of a $2 \times 2$ matrix.

\section{Applications to Complete Graphs} \label{unweighted-apps}

In this section, we demonstrate the applicability of Lemma \ref{rank-one-update} for enumerating spanning trees in complete graphs, complete bipartite graphs, and complete multipartite graphs.  Formal definitions for each of these families of graphs will be given as we progress through this section, but examples of the complete graph $K_5$, the complete bipartite graph $K_{3,4}$, and the complete multipartite graph $K_{2,3,4}$ are shown in Figure \ref{complete-graph-examples}. 

\begin{figure}[H]
\begin{center}

\begin{tikzpicture}
\foreach \t in {0,72,144,216,288}{
\draw[fill=black] (\t:1) circle (.08);
}
\draw (0:1) -- (72:1) -- (144:1) -- (216:1) -- (288:1) -- (0:1);
\draw (0:1) -- (144:1) -- (288:1) -- (72:1) -- (216:1) -- (0:1);
\end{tikzpicture}
\qquad 
\begin{tikzpicture}
\foreach \t in {(0,1.75), (1,1.75), (2,1.75), (3,1.75), (.5,0), (1.5,0), (2.5,0)}{
\draw[fill=black] \t circle (.08);
}
\foreach \s in {(0,1.75), (1,1.75), (2,1.75), (3,1.75)}{

\foreach \t in {(.5,0), (1.5,0), (2.5,0)}{
\draw \s -- \t;
}}

\end{tikzpicture}
\qquad
\begin{tikzpicture}

\draw[fill=black] (0,0.75) circle (.08);
\draw[fill=black] (.75,0.75) circle (.08);
\draw[fill=black] (1.5,0.75) circle (.08);
\draw[fill=black] (2.25,0.75) circle (.08);

\draw[fill=black] (-0.5,2) circle (.08);
\draw[fill=black] (0,2.5) circle (.08);

\draw[fill=black] (2,2.75) circle (.08);
\draw[fill=black] (2.5,2.25) circle (.08);
\draw[fill=black] (3,1.75) circle (.08);

\foreach \r in {(0,0.75), (.75,0.75), (1.5,0.75), (2.25,0.75)}{
	\foreach \s in {(-0.5,2), (0,2.5)}{
	\draw \r -- \s;
	}
	\foreach \t in {(2,2.75), (2.5,2.25), (3,1.75)}{
	\draw \r -- \t;
	}
}
\foreach \s in {(-0.5,2), (0,2.5)}{
	\foreach \t in {(2,2.75), (2.5,2.25), (3,1.75)}{
		\draw \s -- \t;
	}
}

\end{tikzpicture}
\end{center}
\caption{The complete graph $K_5$ (left), the complete bipartite graph $K_{3,4}$ (center), and the complete multipartite graph $K_{2,3,4}$ (right).}
\label{complete-graph-examples}

\end{figure}
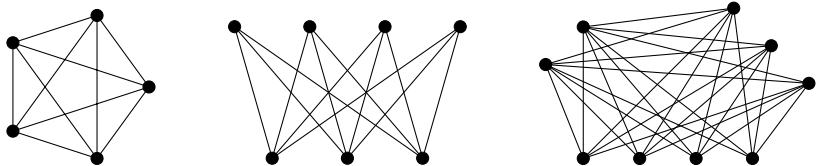

The results presented here are all known in the literature, but the proofs seem new and more elementary than the existing proofs that use Kirchoff's matrix-tree theorem.

For any positive integer $n$, the \textbf{complete graph} $K_n$ is the graph with $n$ vertices, all of which are connected by edges.  The following result, known as Cayley's formula \cite{Cayley}, was first proved by Borchardt \cite{Borchardt} and is widely studied in enumerative combinatorics---for instance, see \cite[\S5.3]{EC2} or \cite[\S33]{Aigner-Ziegler}.

\begin{proposition}
The number of spanning trees in the complete graph $K_n$ is $n^{n-2}$.  
\end{proposition}
\begin{proof}
The Laplacian matrix $L(K_n)$ has entries of $n-1$ on the diagonal and $-1$ in all off-diagonal entries.  Therefore, $L(K_n) + \mathbf{1}_n \mathbf{1}_n^T = nI_n$, where $I_n$ is the $n \times n$ identity matrix.  By Lemma \ref{rank-one-update}, $n^2 \tau(K_n) = \det(n I_n) = n^n$. 
\end{proof}

For every pair of positive integers $m$ and $n$, the \textbf{complete bipartite graph} $K_{m,n}$ has a vertex set partitioned into two sets, $V_1$ and $V_2$ of sizes $m$ and $n$, respectively, where every vertex in $V_1$ is connected to every vertex in $V_2$, but there are no edges among the vertices in $V_1$ or $V_2$.  The next result is a special case of Proposition~\ref{prop:multipartite}, but its statement is clearer and its proof is more straightforward, so it warrants presenting this result first.  Several proofs of this result can be found in the literature, for instance by Hartsfield and Werth \cite{Hartsfield-Werth} and Scoins \cite{Scoins}.

\begin{proposition}
The number of spanning trees in the complete bipartite graph $K_{m,n}$ is $m^{n-1}n^{m-1}$.
\end{proposition}
\begin{proof}
Partition the vertex set of $K_{m,n}$ as $V_1 \sqcup V_2$ with $|V_1| = m$ and $|V_2| = n$ so the Laplacian matrix of $K_{m,n}$ has the block form 
$$
L(K_{m,n}) = \begin{bmatrix} nI_m & -\mathbf{1}_{m,n} \\ -\mathbf{1}_{n,m} & mI_n \end{bmatrix}.
$$
Let $\mathbf{1}_{V_1} \in \mathbb{R}^{m+n}$ be the indicator vector for vertices in $V_1$, meaning each entry indexed by a vertex in $V_1$ is $1$ and all other entries are $0$.  Similarly, let $\mathbf{1}_{V_2}$ be the indicator vector for vertices in $V_2$. Then $L(K_{m,n}) + \mathbf{1}_{V_2}\mathbf{1}_{V_1}^T$ has the upper-triangular block form 
$$
L(K_{m,n}) + \mathbf{1}_{V_2}\mathbf{1}_{V_1}^T = \begin{bmatrix} nI_m & -\mathbf{1}_{m,n} \\ \mathbf{0}_{n,m} & mI_n \end{bmatrix}.
$$
Therefore, by Lemma \ref{rank-one-update}, we have
\begin{eqnarray*}
m \cdot n \cdot \tau(K_{m,n}) &=& 
|V_1|\cdot|V_2| \cdot \tau(K_{m,n}) \\[0.5em] &=& \det\left(L(K_{m,n}) + \mathbf{1}_{V_2}\mathbf{1}_{V_1}^T \right) \\[0.5em]
&=& \det(nI_m) \cdot \det(mI_n) \\[0.5em]
 &=& n^m \cdot n^m.
\end{eqnarray*}
\end{proof}

For every $k$-tuple of positive integers $n_1,\ldots,n_k$, the \textbf{complete multipartite graph} $K_{n_1,\ldots,n_k}$ has a vertex set that can be partitioned into $k$ disjoint sets $V_1, \ldots, V_k$ with $|V_i| = n_i$ for $i \in \{1,\ldots,k\}$ such that $v \in V_i$ and $w \in V_j$ form an edge in $G$ if and only if $i \neq j$.  The following generalization of Cayley's formula has several different proofs by Austin \cite{Austin}, Lewis \cite{Lewis}, and Onodera \cite{Onodera}.  

\begin{proposition}\label{prop:multipartite} 
Let $n_1,\ldots,n_k$ be positive integers and let $n = n_1+\cdots+n_k$.  The number of spanning trees in the complete multipartite graph $K_{n_1,\ldots,n_k}$ is given by   $$\tau(K_{n_1,n_2,\ldots,n_k}) = n^{k-2} \prod_{i=1}^k (n-n_i)^{n_i-1}.$$
\end{proposition}

\begin{proof}
For simplicity, let $K = K_{n_1,n_2,\ldots,n_k}$ and let $V = V(K) = \{v_1,v_2,\ldots,v_n\}$ be ordered so that $v_1,\ldots,v_{n_1}$ make up $V_1$, $v_{n_1+1},\ldots,v_{n_1+n_2}$ make up $V_2$, and so on.  Then the Laplacian matrix of $K$ has blocks of the form $(n-n_i)I_{n_i}$ on its diagonal and all other entries equal to $-1$.  This means $L(K) + \mathbf{1}_{V, V} = L(K) + \mathbf{1}_V \mathbf{1}_{V^T}$ is a block diagonal matrix whose diagonal blocks have the form $(n-n_i)I_{n_i} + \mathbf{1}_{n_i}\mathbf{1}_{n_i}^T$.  By the matrix determinant lemma, 
\begin{eqnarray*}
\det\left((n-n_i)I_{n_i} + \mathbf{1}_{n_i}\mathbf{1}_{n_i}^T\right) &=& \det\left((n-n_i)I_{n_i}\right) \left(1 + \mathbf{1}_{n_i}^T \left((n-n_i)I_{n_i}\right)^{-1} \mathbf{1}_{n_i}\right) \\[0.5em]
&=& (n-n_i)^{n_i} \left(1 + \frac{1}{n-n_i} \mathbf{1}_{n_i}^T \mathbf{1}_{n_i}\right) \\[0.5em]
&=& (n-n_i)^{n_i} \left(1 + \frac{n_i}{n-n_i} \right) \\[0.5em]
&=& (n-n_i)^{n_i-1}\cdot n.
\end{eqnarray*}
Therefore, by Lemma \ref{rank-one-update}, 
\begin{eqnarray*}
n^2 \tau(G) &=& \det\left(L(K) + \mathbf{1}_V \mathbf{1}_{V^T}\right) \\[0.5em]
&=& \prod_{i=1}^k \det\left((n-n_i)I_{n_i}\right) \\[0.5em]
&=& \prod_{i=1}^k \left( (n-n_i)^{n_i-1}\cdot n \right)\\[0.5em]
&=& n^k  \prod_{i=1}^k  (n-n_i)^{n_i-1}.
\end{eqnarray*}
\end{proof}

\section{Application to Ferrers graphs}

In this section, we demonstrate the applicability of Lemmas \ref{rank-one-update} and \ref{schur-complement} for enumerating spanning trees in  a family of bipartite graphs corresponding to integer partitions called \emph{Ferrers graphs}.  

A \textbf{partition} of a positive integer $s$ is a weakly decreasing list of positive integers that sum to $s$.  For example, $(4,4,3,2,1)$ is a partition of $14$.  We write $\lambda = (\lambda_1, \ldots, \lambda_m)$ to denote the parts of a partition $\lambda$.  To any partition, there is an associated \textbf{Ferrers diagram}, which is a stack of left-justified boxes with $\lambda_1$ boxes in the first row, $\lambda_2$ boxes in the second row, and so on. Finally, to any Ferrers diagram there is an associated \textbf{Ferrers graph}, whose vertices are indexed by the rows and columns of the Ferrers diagram with an edge if there is a box in the corresponding position.  The Ferrers diagram and corresponding Ferrers graph associated to the partition $\lambda = (4,4,3,2,1)$ are shown in Figure~\ref{example-ferrers}.

\begin{figure}[H]
\begin{tabular}{>{\centering\arraybackslash}m{.45\textwidth}>{\centering\arraybackslash}m{.45\textwidth}}
\begin{tikzpicture}[scale=.6]
\draw (0,0) -- (4,0);
\draw (0,-1) -- (4,-1);
\draw (0,-2) -- (4,-2);
\draw (0,-3) -- (3,-3);
\draw (0,-4) -- (2,-4);
\draw (0,-5) -- (1,-5);

\draw (0,0) -- (0,-5);
\draw (1,0) -- (1,-5);
\draw (2,0) -- (2,-4);
\draw (3,0) -- (3,-3);
\draw (4,0) -- (4,-2);

\draw (-.5,-.5) node {$r_1$};
\draw (-.5,-1.5) node {$r_2$};
\draw (-.5,-2.5) node {$r_3$};
\draw (-.5,-3.5) node {$r_4$};
\draw (-.5,-4.5) node {$r_5$};

\draw (0.5,0.5) node {$c_1$};
\draw (1.5,0.5) node {$c_2$};
\draw (2.5,0.5) node {$c_3$};
\draw (3.5,0.5) node {$c_4$};
\end{tikzpicture}
&
\begin{tikzpicture}[scale=1]
\foreach \t in {(0,0), (1,0), (2,0), (3,0), (4,0), (.5,1.5), (1.5,1.5), (2.5,1.5), (3.5,1.5)}{
	\draw[fill=black] \t circle (.1);
}
\draw (0,0) node[anchor = north] {$r_1$};
\draw (1,0) node[anchor = north] {$r_2$}; 
\draw (2,0)  node[anchor = north] {$r_3$};
\draw (3,0)  node[anchor = north] {$r_4$};
\draw (4,0)  node[anchor = north] {$r_5$};
\draw (.5,1.5)  node[anchor = south] {$c_1$};
\draw (1.5,1.5) node[anchor = south] {$c_2$}; 
\draw (2.5,1.5) node[anchor = south] {$c_3$};
\draw (3.5,1.5) node[anchor = south] {$c_4$};

\draw (0,0) -- (.5,1.5);
\draw (0,0) -- (1.5,1.5);
\draw (0,0) -- (2.5,1.5);
\draw (0,0) -- (3.5,1.5);

\draw (1,0) -- (.5,1.5);
\draw (1,0) -- (1.5,1.5);
\draw (1,0) -- (2.5,1.5);
\draw (1,0) -- (3.5,1.5);

\draw (2,0) -- (.5,1.5);
\draw (2,0) -- (1.5,1.5);
\draw (2,0) -- (2.5,1.5);

\draw (3,0) -- (.5,1.5);
\draw (3,0) -- (1.5,1.5);

\draw (4,0) -- (.5,1.5);

\end{tikzpicture}
\end{tabular}
\caption{The Ferrers diagram (left) and Ferrers graph (right) corresponding to the partition $(4,4,3,2,1)$.}
\label{example-ferrers}
\end{figure}
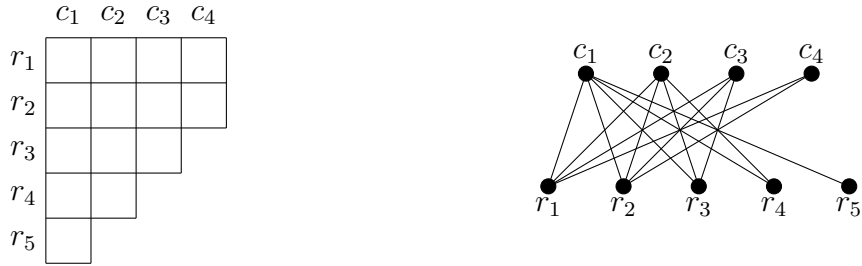

Equivalently,  a \textbf{Ferrers graph} is a bipartite graph $G$ whose vertices can be partitioned as $R \sqcup C$ with $R = \{r_1,\ldots,r_m\}$ (corresponding to the rows of the Ferrers diagram) and $C = \{c_1,\ldots,c_n\}$ (corresponding to the columns) such that 
\begin{enumerate}
\item if $\{r_k,c_{\ell}\}$ is an edge in $G$, then so is $\{r_i,c_j\}$ for any $i \leq k$ and $j \leq \ell$, and
\item $\{r_1,c_n\}$ and $\{r_m,c_1\}$ are edges in $G$.
\end{enumerate}

Ehrenborg and van Willigenburg \cite[Theorem 2.1]{E-VW} found the following beautiful formula counting the number of spanning trees in a Ferrers graph.  

\begin{theorem}\label{thm:EVW}
Let $G$ be a Ferrers graph whose vertices are partitioned as $V(G) = R \sqcup C$ with $R = \{r_1,\ldots,r_m\}$ and $C = \{c_1,\ldots,c_n\}$.  Then $$\tau(G) = \frac{\prod_{v \in V(G)} \deg(v)}{mn} = \prod_{i=2}^m \deg(r_i) \prod_{j=2}^n \deg(c_j).$$
\end{theorem}

We will give a new proof of this result using Lemmas \ref{rank-one-update} and \ref{schur-complement}.  To begin, let us explore some general properties of Laplacians of bipartite graphs through the lens of Lemma~\ref{rank-one-update}. 

Let $G$ be a connected bipartite graph whose vertices are partitioned as $V(G) = R \sqcup C$.  The Laplacian matrix of $G$ can be partitioned as a block matrix with the form 

$$
L(G) = \begin{bmatrix} D_R & B \\ B^T & D_C\end{bmatrix},
$$ \smallskip

where $D_R$ (respectively, $D_C$) is an $m \times m$ (respectively, $n \times n$) diagonal matrix whose diagonal entries encode the degrees of the vertices in $R$ (respectively, $C$).  The matrix $B$ is an $m \times n$ matrix with rows indexed by vertices $r \in R$, columns indexed by vertices $c \in C$, and $B(r,c) = -1$ if $\{r,c\}$ is an edge and a $0$ otherwise. 

Now we consider a rank-one update to the Laplacian matrix that will be useful in interpreting Theorem \ref{thm:EVW}.  Consider the matrix 

$$
M(G) = L(G) + \mathbf{1}_C \mathbf{1}_R^T = L(G) + \begin{bmatrix} \mathbf{0}_{m,m} & \mathbf{0}_{m,n} \\ \mathbf{1}_{n,m} & \mathbf{0}_{n,n} \end{bmatrix},
$$ \smallskip

where $\mathbf{1}_R$ and $\mathbf{1}_C$ are the indicator vectors for $R$ and $C$ in $\mathbb{R}^{V(G)}$.  We can decompose $M(G)$  as 

$$M(G) = \begin{bmatrix} D_R & B \\ B^{op} & D_C \end{bmatrix},$$ \smallskip

where $D_R$, $B$, and $D_C$ are defined as in $L(G)$ and $B^{op} = B^T + \mathbf{1}_{n,m}$.  As a consequence of Lemma \ref{rank-one-update}, $\det(M(G)) = mn \cdot \tau(G)$.

On the other hand, $B(r,c) = -1$ if and only if $\{r,c\} \in E(G)$, $B^{op}(c,r') = 1$ if and only if $\{r',c\} \notin E(G)$, and the entries of $B$ and $B^{op}$ equal zero otherwise. Applying Lemma~\ref{schur-complement} to $M(G)$ and $D_C$, we see that 

$$
\det(M(G)) = \det(D_C) \cdot \det(D_R - BD_C^{-1}B^{op}).
$$ \smallskip

We can compute the entries of $S: = D_R - BD_C^{-1}B^{op}$ explicitly. Let $r$ and $r'$ be elements of $R$ that are not necessarily distinct.  The row indexed by $r$ in $B$ has a nonzero entry of $-1$ for each $c \in C$ that is a neighbor of $r$ and the column indexed by $r'$ in $D_C^{-1}B^{op}$ has a nonzero entry of $\frac{1}{\deg(c)}$ for each $c \in C$ that is not a neighbor of $r'$.  Therefore, the entry in row $r$ and column $r'$ of $BD_C^{-1}B^{op}$ is equal to $\sum \frac{-1}{\deg(c)}$, where the sum is over all $c \in C$ that are incident with $r$, but not $r'$. Consequently, the entries on the diagonal of $BD_C^{-1}B^{op}$ are all zero. This proves the following proposition.
\begin{proposition} \label{bipartite-S-matrix}
Let $G$ be a bipartite graph whose vertex set is partitioned as $R \sqcup C$.  Then 
$$
mn \cdot \tau(G) = \prod_{c \in C} \deg(c) \cdot \det(S),
$$
where $S$ is the $m \times m$ matrix with entries given by
$$
S(r,r') = \begin{cases}
\deg(r) & \text{ if } r = r', \\
\displaystyle \sum_{c \in N(r) \setminus N(r')} \frac{1}{\deg(c)} & \text{ otherwise,}
\end{cases}
$$
where $N(v)$ denotes the set of neighbors of a vertex $v$. \qed
\end{proposition}

We can now give a simple proof of Theorem \ref{thm:EVW}.

\begin{proof}[Proof of Theorem~\ref{thm:EVW}]
Let $G$ be a Ferrers graph as described in the theorem statement and observe that the vertices in $R$ are naturally ordered such that $N(r_1) \supseteq N(r_2) \supseteq \cdots \supseteq N(r_m)$.  This means $N(r_i) \setminus N(r_j) = \emptyset$ for every $i > j$, and hence the corresponding entries of the $S$ matrix in Proposition~\ref{bipartite-S-matrix} satisfy $S(r_i,r_j) = 0$.  Thus, the $S$ matrix is upper triangular with diagonal entries $\deg(r)$ for $r \in R$, which means $\det(S) = \prod_{r \in R} \deg(r)$.  The result follows by Proposition \ref{bipartite-S-matrix}.
\end{proof}

\section{Application to Threshold graphs}

In this section, we demonstrate the applicability of Lemmas \ref{rank-one-update} and \ref{schur-complement} for enumerating spanning trees in an important and well-studied family of graphs called \emph{threshold graphs}; see \cite{Mahadev-Peled}.

A graph $G$ on $n$ vertices is a \textbf{threshold graph} if its vertices can be ordered $v_1,\ldots,v_n$ in such a way that if $\{v_k,v_\ell\} \in E(G)$ for some $1 \leq k < \ell \leq n$, then $\{v_i,v_\ell\} \in E(G)$ whenever $i < k$ and $\{v_k,v_j\} \in E(G)$ for all $j < \ell$. For a threshold graph $G$, we consider a special vertex $v_t$, where $t$ is taken to be the largest index such that $\{v_i,v_t\} \in E(G)$ for all $i < t$. Because $G$ is threshold, this means $\{v_i,v_j\} \in E(G)$ for all $1 \leq i<j \leq t$.\footnote{For readers who are familiar with the definition of threshold graphs starting from an initial vertex and inductively adding dominating or isolated vertices, the vertex $v_t$ in our definition is the initial vertex, the vertices $v_1,\ldots,v_{t-1}$ are the dominating vertices in $G$ (where $v_{t-1}$ is the first dominating vertex, $v_{t-2}$ the second, and so on), and the vertices $v_{t+1}, \ldots, v_n$ are the isolated vertices (where $v_{t+1}$ is the first isolated vertex, $v_{t+2}$ the second, and so on).} An example threshold graph is illustrated in Figure~\ref{example-threshold}.

\begin{figure}[H]
\begin{center}
\scalebox{0.9}{
\begin{tikzpicture}
\draw [fill] (0,2.4) circle [radius=0.12];
\node [above left] at (0,2.4) {$6$};
\draw [fill] (1.6,2.4) circle [radius=0.12];
\node [above right] at (1.6,2.4) {$1$};
\draw [fill] (2.6,1.2) circle [radius=0.12];
\node [right] at (2.6,1.2) {$ \, 2$};
\draw [fill] (1.6,0) circle [radius=0.12];
\node [below right] at (1.6,0) {$3$};
\draw [fill] (0,0) circle [radius=0.12];
\node [below left] at (0,0) {$4$};
\draw [fill] (-1,1.2) circle [radius=0.12];
\node [left] at (-1,1.2) {$5 \, $};
\draw [thick] (1.6,2.4) -- (2.6,1.2);
\draw [thick] (1.6,2.4) -- (1.6,0);
\draw [thick] (1.6,2.4) -- (0,0);
\draw [thick] (1.6,2.4) -- (-1,1.2);
\draw [thick] (1.6,2.4) -- (0,2.4);
\draw [thick] (2.6,1.2) -- (1.6,0);
\draw [thick] (2.6,1.2) -- (0,0);
\draw [thick] (2.6,1.2) -- (-1,1.2);
\draw [thick] (2.6,1.2) -- (0,2.4);
\draw [thick] (1.6,0) -- (0,0);
\draw [thick] (1.6,0) -- (-1,1.2);

\end{tikzpicture}}
\end{center}
\caption{An example threshold graph on six vertices with $t = 4$.}
\label{example-threshold}
\end{figure}
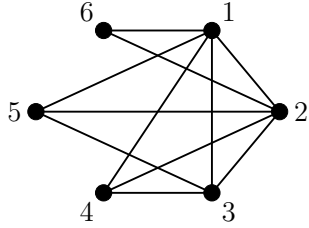

In this article, we are only interested in connected threshold graphs since disconnected graphs do not have any spanning trees.  An important property of threshold graphs is that a threshold graph is connected if and only if its highest index vertex, $v_n$, is not an isolated vertex.  In other words, as long as $v_n$ has neighbors, the graph is connected.

As with Ferrers graphs, spanning trees in threshold graphs can also be counted by a beautiful formula, which appears several places in the literature, including proofs by Chestnut and Fishkind \cite{Chestnut-Fishkind}, Hammer and Kelmans \cite{Hammer-Kelmans}, and Merris \cite{Merris}.    

\begin{theorem}
Let $G$ be a connected threshold graph on $n$ vertices, and let $t \leq n$ be the largest index such that $\{v_i,v_t\} \in E(G)$ for all $i < t$.  Then 
\begin{equation} \label{threshold-enumerator}
\tau(G) =  \prod_{i=2}^{t-1} \left( \deg(v_i)+1 \right) \prod_{i=t+1}^n \deg(v_i).
\end{equation}
\end{theorem}

Before we give a new proof of this theorem, let us point out a familiar example.  In the special case that $t=n$, the threshold graph in question is the complete graph $K_n$.  In this case, each vertex has degree $n-1$, so the expression in equation~\eqref{threshold-enumerator} simplifies to $\tau(G) = \prod_{i=2}^{n-1}n = n^{n-2}$ as we saw in Cayley's formula. 

\begin{proof}
We begin with a few observations:  First, since $G$ is connected, $\deg(v_n) > 0$ as noted above, which implies that $\deg(v_i) > 0$ for all $i$. Next, recall that $\{v_i,v_j\} \in E(G)$ for all $i < j \leq t$ and, because $t$ was chosen to be the largest index such that $\{v_i,v_t\} \in E(G)$ whenever $i < t$, it follows that $\{v_k,v_\ell\} \notin E(G)$ whenever $t \leq k < \ell$.  Therefore, we know that the Laplacian matrix $L(G)$ can be partitioned into blocks as 
$$
L(G) = \begin{bmatrix} A & B \\ B^T & D \end{bmatrix},
$$
where $A$ is the $(t-1) \times (t-1)$ matrix whose diagonal entries are $\deg(v_i)$ for $i < t$ and whose off-diagonal entries are all $-1$, $D$ is the $(n-t+1) \times (n-t+1)$ diagonal matrix whose diagonal entries are $\deg(v_j)$ for $j \geq t$, and $B$ records the adjacencies $\{v_i,v_j\}$ with $i < t$ and $j \geq t$.

Now let $U = \{v_1,\ldots,v_{t-1}\}$ and $V = \{v_1,\ldots,v_n\}$, and consider the matrix $L(G) + \mathbf{1}_V\mathbf{1}_U^T$, where $\mathbf{1}_V$ is the $|V| \times 1$ vector of ones and $\mathbf{1}_U$ is the indicator vector for $U$ in $V$.  By Lemma \ref{rank-one-update}, $$\det(L(G) + \mathbf{1}_V\mathbf{1}_U^T) = (t-1) \cdot n \cdot \tau(G).$$  On the other hand, 
$$
L(G) + \mathbf{1}_V\mathbf{1}_U^T = \begin{bmatrix} A' & B \\ B^{op} & D \end{bmatrix}, 
$$
where $A' = A+\mathbf{1}_{U \times U}$ is a diagonal matrix with diagonal entries $\deg(v_i)+1$ for $i < t$ and $B^{op} = B^T + \mathbf{1}_{(V\setminus U) \times U}$.  Note that, for all $i < t$ and $j \geq t$, $B$ has a nonzero entry $B(i,j) = -1$ if and only if $\{v_i,v_j\} \in E(G)$ and $B^{op}$ has a nonzero entry $B^{op}(j,i) = 1$ if and only if $\{v_i,v_j\} \notin E(G)$. Since the diagonal matrix $D$ is invertible by our assumption that $\deg(v_i) > 0$ for all $i$, we can apply Lemma~\ref{schur-complement} to get that 
$$\det(L(G) + \mathbf{1}_V\mathbf{1}_U^T) = \det(D) \det(A'-BD^{-1}B^{op}).$$

As in the proof of Theorem~\ref{thm:EVW}, we can explicitly compute the entries of $BD^{-1}B^{op}$. For $i < j < t$, the entry in the row indexed by vertex $v_i$ and column indexed by $v_j$ is equal to $\sum \frac{-1}{\deg(v_k)}$, where the sum is over all $v_k$ with $k \geq t$ such that $v_k \in N(v_i)$ and $v_k \notin N(v_j)$.
Because $G$ is a threshold graph, the set of such $v_k$ is empty when $i < j$, which means $A'-BD^{-1}B^{op}$ is an upper-triangular matrix with diagonal entries $\deg(v_i) + 1$ for $1 \leq i < t$.  Since $D$ is a diagonal matrix with diagonal entries $\deg(v_k)$ for $k \geq t$, we know that 
$$
(t-1) \cdot n \cdot \tau(G) = \det(L(G) + \mathbf{1}_V\mathbf{1}_U^T) = \prod_{i=1}^{t-1} \left( \deg(v_i)+1 \right) \prod_{i=t}^n \deg(v_i).
$$
To complete the proof, we observe that because $G$ is threshold and $\deg(v_n) > 0$, it must be the case that $\{v_1,v_n\} \in E(G)$, hence $\{v_1,v_j\} \in E(G)$ for all $1 < j \leq n$.  Therefore, $\deg(v_1)+1 = n$.  Similarly, by our choice of $t$, we see that $\deg(v_t) = t-1$. The result now follows by cancelling appropriate terms from both sides. 
\end{proof}

\section{Conclusion}
We have shown that the matrix determinant lemma and method of Schur complements can be used to simplify spanning tree enumeration for several families of graphs.  While they may not shed as much light on spanning tree enumeration as a bijective proof would, they do give simpler, more direct proofs than other purely linear-algebraic approaches using the matrix-tree theorem.  Our methods also extend naturally to weighted spanning tree enumeration \cite{Klee-Stamps-weighted}. Which other families of graphs are amenable to having their spanning trees counted by Lemma \ref{rank-one-update}? We hope you can tell us!

\section*{Acknowledgments}
Steven Klee's research was supported by NSF grant DMS-1600048.  Matthew Stamps is grateful to Isabella Novik and the Department of Mathematics at the University of Washington for hosting him during the time this research was conducted.

\end{document}